\documentclass{tran-l}

 \newtheorem{thm}{Theorem}[section]
 \newtheorem{cor}[thm]{Corollary}
 \newtheorem{lem}[thm]{Lemma}
 \newtheorem{prop}[thm]{Proposition}
 \theoremstyle{definition}
 \newtheorem{defn}[thm]{Definition}
 \theoremstyle{remark}
 
 \newtheorem{exmp}[thm]{Example}
 \numberwithin{equation}{section}
 \newtheorem{oprob}[thm]{Open Problem}
 \newtheorem{conj}[thm]{Conjecture}
\begin{document}
\title[On Vertically-Recurrent Matrices]
 {On Vertically-Recurrent Matrices and Their Algebraic Properties}

\author{Hossein Teimoori Faal}

\address{Department of Mathematics and Computer Science, Allameh Tabataba'i University, Tehran, Iran}

\email{Hossein.teimoori@atu.ac.ir}



\date{}

%\dedicatory{Dedicated to Professor Hydar Radjavi on occasion of
%His $70$th Birthday}

\commby{}

%%% ----------------------------------------------------------------------
\begin{abstract}
In this paper, we first introduce the new class of
vertically-recurrent matrices, using a generalization of "the
Hockey stick and Puck theorem" in Pascal's triangle. 
Then, we give
an interesting formula for the lower triangular decomposition of these matrices. We
also deal with the $m$-th power of these matrices in some special
cases. Furthermore, we present two important applications of these
matrices for decomposing \emph{admissible matrices} and matrices
which arise in the theory of \emph{ladder networks}. Finall,y we pose
some open problems and conjectures about these new kind of matrices.
\end{abstract}
%%% ----------------------------------------------------------------------
\maketitle
%%% ----------------------------------------------------------------------

\section{INTRODUCTION}

In theory of linear algebra, the general problem of the classification of the
\emph{integral matrices} which have simpler decompositions is the key in many 
theoretical and applied fields. In this paper, we intend to find the \emph{lower-triangular}  decomposition for
some new kind of matrices which we call them  \emph{matrices with
vertically recurrence relation} into Toeplitz block matrices. 
\\
Fortunately, several interesting classes of the integer-valued Toeplitz matrices can be 
nicely factorized into Pascal matrices [3]. The important point in finding this matrix decomposition is the well-known property of the Pascal triangle which is called \emph{the
hockey stick and puck theorem}. Using the generalization of the
above  theorem, one can construct a new linearly recurrence
relation which we call it the \emph{generalized hockey stick and puck}
theorem. This later one is the principle of the \emph{multiplicative
decomposition}
of the above matrices .\\
In section $2$, we start by the definition of the matrix with
\emph{vertically-recurrent} relation associated with the sequence
$\Lambda=\{ \lambda_{n} \}_{n\geq 0}$, then we investigate some of
it's properties, specially its multiplicative decomposition and
also we find its inverse matrix. In section $3$, we deal with the
power of this matrix and find it's associated sequence $\Lambda=\{
\lambda_{n} \}_{n\geq 0}$. Furthermore, we present two important
applications of this new kind of matrix for factorization of
\emph{admissible matrices} and in \emph{ladder networks}. Finally, we propose 
some open problems and conjectures about these matrices.
\newpage 

\section{The Vertically-Recurrent Matrices}

We start by motivating the main idea behind these new \emph{integral} matrices. Consider
the two dimensional 
\emph{linear recurrence} relation
among entries of the well-known  \emph{Pascal's triangle}, as follows
$$
\begin{array}{ccc}
 \overset{\bf u}{\bullet}&&\overset{\bf v}\bullet \\\\
 &  &\overset{\bf w}\bullet
\end{array}
$$
\begin{center}
Figure 1. $  w=u+v $
\end{center}
Now if we consider the two consecutive columns of the left-justified Pascal's
triangle; i.e,. the $k-1$ and the $k$th columns, we have
$$
\begin{array}{ccc}
 \overset{\bf u_l}\bullet&&\\
 \overset{\bf u_{l-1}}\bullet&&\overset{\bf v_{l-1}=w_{l-1}}\bullet\\
 \overset{\bf u_{l-2}}\bullet&&\overset{\bf v_{l-2}=w_{l-2}}\bullet\\
 \vdots &&\vdots\\
 \overset{\bf u_2}\bullet&&\overset{\bf v_2=w_2}\bullet\\
\overset{\bf u_1}\bullet&&\overset{\bf v_1=w_1}\bullet\\
 &&\overset{\bf w=w_0}\bullet\\
\end{array}
$$
\begin{center}
Figure 2. Hockey Stick and Puck Theorem
\end{center}
Therefore, at any step, after computing the $w_i$'s with respect
to $u_i$'s and $v_i$'s, fix the $u_i$ and just rewrite $v_i$, as
the
 next $w_i$, with respect to $u_{i+1}$ and $v_{i+1}$. Continuing this
 process until to get the main diagonal. It can be easily seen that $w$ is
 expressible as the sum of the entries $u_1,u_2,\ldots,u_l$. More precisely, we have
\begin{eqnarray}
w_0=w&=&u_1+v_1\nonumber\\
w_1=v_1&=&u_2+v_2\nonumber\\
&\vdots&\nonumber\\
w_{l-2}=v_{l-2}&=&u_{l-1}+v_{l-1}\nonumber\\
w_{l-1}=v_{l-1}&=&u_{l}\nonumber
\end{eqnarray}
and consequently,
\begin{equation}\label{2}
w=u_1+u_2+\cdots+u_l. 
\end{equation}
Now, if we translate the above relation into the language of recurrence relations, we
obtain
\begin{equation}\label{Vert-Rec1}
a_{n,k}=\sum_{l=k-1}^{n-1}{a_{l,k-1}},\hspace{0.5cm} (n \geq k \geq 1).
\end{equation}
The above equality is known as the \emph{vertically-recurrent} relation.
In special case for the left-justified Pascal's triangle; that is
$a_{n,k} ={n-1 \choose k-1}$,
we get the well-known \emph{hockey stick
and puck} theorem [4] as follows ( see Figure 2 )
\begin{equation}
{n\choose k}=\sum_{l=k-1}^{n-1}{{l\choose k-1}},\hspace{1cm}
(n \geq k \geq 1). 
\end{equation}
Thus, it is natural to generalize the relation (\ref{Vert-Rec1}) associated
with an arbitrary sequence $\Lambda=\{\lambda_n \}_{n\geq0 }$,
$\lambda_0=1$, as follows:
\begin{equation}\label{GenVertRec1}
a_{n,k}=\sum_{l=k-1}^{n-1}{\lambda_{n-1-l}a_{l,k-1}},\hspace{1cm}
(n \geq k\geq 1).
\end{equation}
We call the equation (\ref{GenVertRec1}) \emph{the generalized hockey stick and
puck theorem} . Now we are at the position to define our new class of matrices that we call them \emph{vertically-recurrent} matrices. 
\\ 
For combinatorial reasons, we mainly concentrate on the class of matrices in which their associate sequences have only integer values. 
\begin{defn}
Suppose $n$ and $k$ are positive integers. We define the 
\emph{vertically-recurrent} matrix
$V_{n}[{\bf \Lambda}]$
associated with the sequence
${\bf \Lambda}=\{\lambda_n \}_{n\geq0 }$ , $\lambda_0=1$, of order
$(n+1)\times(n+1)$ in the following form:
$$
(V_{n}[{\bf \Lambda}])_{ij}= \left\{
\begin{array}{ccc}
\lambda_{i} &\hspace{1cm}\mbox{if\ } i\geq0, j=0, \\
a_{i,j}&\hspace{1cm}\mbox{if\ } j\geq i \geq 1,\\
0 &\hspace{.4cm}\mbox{if\ } i<j,
\end{array}
\right.
$$
in which the entries $a_{i,j}$'s satisfy the relation (\ref{GenVertRec1}).
\end{defn}

\begin{exmp} 
	
 For ${\bf\Lambda}=\{\lambda_{n}=1\}_{n\geq0}$, we have:\\
$$V_{3}[{\bf \Lambda}]= \left[
\begin{array}{cccc}
1 & 0 & 0 & 0\\ 1 &1& 0 & 0\\ 1 & 2 &1& 0\\1 &3&3&1
\end{array}
\right],
$$
\\
where the above matrix is call the Pascal matrix $P_{n}$, for 
$n=3$
(see [5]).

\end{exmp} 

\begin{exmp} 

For ${\bf\Lambda}=\{\lambda_{n}=2^n\}_{n\geq0}$, we have:\\
$$ V_{3}[{\bf \Lambda}]= \left[
\begin{array}{cccc}
1 & 0 & 0 & 0\\ 2 &1& 0 & 0\\ 4 & 4 &1& 0\\8 &12&6&1
\end{array}
\right].
$$
\\
We note that the above matrix is the 
\emph{Pascal functional matrix}  $P_n[x]$
for $n=3$ and $x=2$ (see \cite{m4}).

\end{exmp}

Next, we start to obtain a \emph{multiplicative decomposition} of
$V_{n}[{\bf \Lambda}]$. To do this, we first define the lower
triangular \emph{Teoplitz} matrix by $T_n[{\bf
\Lambda}]=[\lambda_{i-j}]_{0\leq i\leq j\leq n}$ and the
Teoplitz-block
matrix $\overline{T}_{k}[{\bf \Lambda}]$, as follows\\
\begin{eqnarray}
\overline{T}_{k}[{\bf \Lambda}]=\left[
\begin{array}{cccc}
 I_k&0\\
 0&T_{n-k}[{\bf \Lambda}]
\end{array}\right].\nonumber
\end{eqnarray}
By convention, $\overline{T}_{0}[{\bf \Lambda}]= T_n$ and
$\overline{T}_{n}[{\bf \Lambda}]= I_{n+1}$, where $I_{n+1}$ is the
identity matrix of order $n+1$.
\begin{thm}\label{MultDecomp1}
Suppose $n$ is a natural number. Then, we have
$$
V_{n}[{\bf \Lambda}]= T_{n}[{\bf \Lambda}]([1]\oplus V_{n-1}[{\bf
\Lambda}]),
$$
in which the symbol $\oplus$ denotes the direct sum of two
matrices.
\end{thm}
\begin{proof}
For each $i$ and $j$ with $i\geq j\geq0$, since the $(i,j)$-entry
of $[1]\oplus V_{n-1}[{\bf \Lambda}]$ is $(V_{n}[{\bf
\Lambda}])_{i-1,j-1}$, from the definition of matrix product and
the relation (0.4), we get
\begin{eqnarray}
(T_{n}[{\bf \Lambda}]([1]\oplus
V_{n-1}[{\bf \Lambda}]))_{i,j}
&=&\sum_{l=j}^{i}(T_{n}[{\bf \Lambda}])_{i,l}(V_{n}[{\bf \Lambda}])_{l-1,j-1}\nonumber\\
&=&\sum_{l=j-1}^{i-1}(T_{n}[{\bf \Lambda}])_{i,l+1}(V_{n}[{\bf \Lambda}])_{l,j-1}\nonumber\\
&=&\sum_{l=j-1}^{i-1}\lambda_{i-l-1}a_{l,j-1}\nonumber\\
&=&a_{i,j}=(V_{n}[{\bf \Lambda}])_{i,j}\nonumber
\end{eqnarray}
\end{proof}
Now, as an immediate consequence of Theorem \ref{MultDecomp1}, we have the
following results:
\begin{cor}
~\vspace{.2cm}
\begin{itemize}
\item[(i).] $V_{n}[{\bf \Lambda}]=\overline{T}_{n}[{\bf
\Lambda}]\overline{T}_{n-1}[{\bf
\Lambda}]\cdots\overline{T}_{1}[{\bf \Lambda}].$
\item[(ii).]$
V_{n}^{-1}[{\bf\Lambda}]=\overline{T}_{1}^{-1}[{\bf
\Lambda}]\overline{T}_{2}^{-1}[{\bf
\Lambda}]\cdots\overline{T}_{k}^{-1}[{\bf \Lambda}].$
\end{itemize}
\end{cor}
In [3], the Teoplitz matrices with 
\emph{integer entries} are
investigated. Considering these matrices, we can calculate the
inverse of $\overline{T}_{n}[{\bf \Lambda}]$ by means of the
Pascal functional matrices [3] in some important special
cases. We leave the general case as an open question.\\
{\bf Case 1}. Let ${\bf\Lambda}=\{\lambda_n=\lambda\}_{n\geq0}$, then we
clearly have
\begin{equation}
\overline{T}_{n}[{\bf\Lambda}]=\lambda S_n[1],
\end{equation}
where
$$
(S_n[x])_{ij}= \left\{
\begin{array}{cc}
x^{i-j} &\hspace{.5cm}\mbox{if\ } i\geq j\geq0, \\
0&\hspace{.1cm}\mbox{if\ } i<j.
\end{array}
\right.
$$
We also have $S_n[x]=P_{n,1}[x]P_n[-x]$, in which the matrices
$P_n[x]$ and $P_{n,1}[x]$ are Pascal functional and Pascal
k-eliminated functional matrices,  respectively [5,6]. Thus, we get
$$
T_{n}[{\bf\Lambda}]=\lambda P_{n,1}[1]P_n[-1].
$$
{\bf Case 2}. Let ${\bf\Lambda}=\{\lambda_n=\lambda^{n}\}_{n\geq0}$. Now, 
it is clearly the generalization of the above case, since in this
case $T_n[{\bf\Lambda}]=S_n[\lambda]$ and consequently
$$
T_n[{\bf\Lambda}]=
P_{n,1}[\lambda]P_n[-\lambda].
$$
%\end{proof}
\section{The Powers of Vertically-Recurrent Matrices}
If we consider the Pascal functional matrix for the values
$1,2,\ldots,l$, then we observe that all of these matrices are
\emph{vertically-recurrent}. Indeed, for
$P_n[l]$ in general, the associated sequence is
${\bf\Lambda}=\{\lambda_n=l^n\}_{n\geq0}$. On the other hand, we
know that the Pascal matrix $P_n[x]$ has an exponential property
(see [5]) therefore the matrix $P_n[l]$ is just the $l$-th power
of the matrix $P_n[1]$. For the above reason, the following challenging 
question naturally arises in the context of vertically-recurrent matrices. 

\begin{quote}
\emph{If the associated sequence of the matrix $V_n[{\bf\Lambda}]$
is ${\bf\Lambda}=\{\lambda_n\}_{n\geq0}$, then what is the
associated sequence of the matrix $(V_n[{\bf\Lambda}])^m$ with
respect to the sequence $\lambda_n$?}
\end{quote}

In general case ${\bf\Lambda}=\{\lambda_n\}_{n\geq0}$, the above
question is a very challenging problem but we have given an affirmative
answer to some special and interesting cases.
\\
{\bf Case1}. Suppose $V_n[{\bf \Lambda}]$ is a matrix with
constant associated sequence
${\bf\Lambda}=\{\lambda_n=\lambda\}_{n\geq0}$. We have the
following interesting result.
\begin{prop}
{Let $V_n[{\bf\Lambda}]$ be a vertically-recurrent matrix 
with its associated sequence ${\bf\Lambda}=\{\lambda_n = \lambda \}_{n\geq0}$. Then,
the associated sequence of $(V_n[{\bf\Lambda}])^m$ is
$\lambda^m(\frac{\lambda^m-1}{\lambda-1})^{n}$}
\end{prop}
\begin{proof}
{The above statement is equal to prove that the recurrence
relation for entries of $(V_n[{\bf\Lambda}])^m$ is, as follows: }
\end{proof}
\begin{equation}
a_{n,k}=\lambda^m a_{n-1,k-1}+(\frac{\lambda^m-1}{\lambda-1})^n
a_{n-1,k}.
\end{equation}
Since, considering the above identity and mathematical induction
we are able to prove,
$$
a_{n,k}=\sum_{k-1}^{n-1}\big[\lambda^m(\frac{\lambda^m-1}{\lambda-1})\big]^{n-1-l}a_{l,k-1},
$$
Namely, $\lambda^m(\frac{\lambda^m-1}{\lambda-1})^n$ is the
associated sequence for $(V_n[\Lambda])^m$. For proving the
equivalence statement we just need to consider the following
result from [7].

\begin{lem}
Suppose $\alpha, \alpha', \beta$, $\beta'$ are four real numbers.
Also let $A=[a_{ij}]$, $B=[b_{ij}]$ be two lower triangular
matrices where their entries satisfy the following recurrence relations
respectively,
\begin{eqnarray}
&&\left\{
\begin{array}{cc}
a_{n,k}=\alpha a_{n-1,k-1}+\beta a_{n-1,k},&~~~(n\geq k\geq1),\\
a_{n,0}=1& n\geq 0,\\
a_{n,k}=0 & k>n,
\end{array}
\right.\\
&&\left\{
\begin{array}{cc}
b_{n,k}=\alpha' b_{n-1,k-1}+\beta' b_{n-1,k},&~~~(n\geq k\geq1),\\
b_{n,0}=1&n\geq 0,\\
b_{n,k}=0 &k>n.
\end{array}
\right.
\end{eqnarray}
{\it If $AB=[c_{ij}]$ then, there are real numbers
$\alpha''=\alpha\alpha'$ and $\beta''=\beta+\alpha\beta'$, such
that }
\begin{eqnarray}
\left\{
\begin{array}{cc}
c_{n,k}=\alpha'' c_{n-1,k-1}+\beta'' c_{n-1,k},&(n\geq k\geq1),\\
c_{n,0}=\sum_{i=0}^{n}a_{n,i}&n\geq 0,\\
\hspace{-.8cm}c_{n,k}=0~~~~ &k>n.
\end{array}
\right.
\end{eqnarray}
\end{lem}

\begin{proof}
	
Let $C=AB=[c_{n,k}]$. By the definition of the product of two matrices, we conclude that 
$
c_{n,k} = \sum_{l=k}^{n} a_{n,l} b_{l,k} ~(n\geq k \geq 1)
$	
. Hence, this immediately implies that
$
c_{n,0} = \sum_{l=0}^{n} a_{n,l} b_{l,0} = 
\sum_{l=0}^{n} a_{n,l} ~ (n\geq 0)
$
and 
$
c_{n,k} = 0~(k > n)
$
. 
Moreover, we have 
\begin{equation}\label{keyrec1}
c_{n-1,k-1} = 
\sum_{l=k-1}^{n-1} a_{n-1,l} b_{l,k-1}, ~~~
c_{n-1,k} = 
\sum_{l=k}^{n-1} a_{n-1,l} b_{l,k}.  
\end{equation}	
Put 
$
I_{n,k} = \alpha \alpha^{'} c_{n-1,k-1} + (\beta + \alpha \beta^{'})
c_{n-1,k}
$
. Then, we have
 	
\begin{eqnarray}
I_{n,k} & = & \alpha \alpha^{'} \sum_{l=k-1}^{n-1} a_{n-1,l} b_{l,k-1} + (\beta + \alpha \beta^{'}) 
\sum_{l=k}^{n-1} a_{n-1,l} b_{l,k}, \nonumber\\ 
& = & 
\Bigg[
\alpha \alpha^{'} \sum_{l=k-1}^{n-1} a_{n-1,l} b_{l,k-1} 
+ \alpha \beta^{'} \sum_{l=k}^{n-1} a_{n-1,l} b_{l,k}
\Bigg] \nonumber\\  
& + & \beta \sum_{l=k}^{n-1} a_{n-1,l} b_{l,k}, \nonumber\\
& = & \alpha 
\Bigg[
 \sum_{l=k}^{n-1} a_{n-1,l}
\Big( \alpha^{'} b_{l,k-1} 
+  \beta^{'} b_{l,k} \Big) + a_{n-1,k-1} b_{k-1,k-1}
\Bigg] \nonumber\\
& + &  \beta \sum_{l=k}^{n-1} a_{n-1,l} b_{l,k}, \nonumber\\
& = & \alpha \Bigg[ \sum_{l=k}^{n-1} b_{l+1,k} + a_{n-1,k-1}
a_{n-1,k-1} \Bigg] \nonumber\\ 
& + &  \beta \sum_{l=k+1}^{n} a_{n-1,l} b_{l,k} + 
\beta \Big( a_{n-1,k-1} b_{k,k} - a_{n-1,n} b_{n,k} \Big)
\nonumber\\
& = & \sum_{l=k+1}^{n} 
\Big( 
\alpha a_{n-1,l-1} + \beta a_{n-1,l} 
\Big) b_{l,k} + 
\Big( 
\alpha a_{n-1,l-1} + \beta a_{n-1,k} 
\Big) \nonumber\\
& = & 
\sum_{l=k+1}^{n} a_{n,l} b_{l,k} + a_{n,k} \nonumber \\
& = & \sum_{l=k}^{n} a_{n,l} b_{l,k}=c_{n,k}, \nonumber
\end{eqnarray}	
as required.

\end{proof}

Now the equivalence statement is easily proved considering the
above lemma and the mathematical induction.
\\
{\bf Case2.} $\lambda_n=\lambda^n$. In this case, we use the above
lemma again to obtain the following theorem,
\begin{thm}
{Let $V_n[\Lambda]$ be a matrix with vertically recurrent relation
and it's associated sequence $\lambda_n=\lambda^n$. Then, the
associated sequence of $(V_n[\Lambda])^m$ is}
\end{thm}
\begin{equation}
\lambda_n=(\lambda{m})^n.
\end{equation}
\begin{proof}
Considering the same argument in theorem $0.4$, It is necessary to
prove the following recurrent relation for $(V_n[\Lambda])^m$:
$$
a_{n,k}=a_{n-1,k-1}+(2m)\lambda a_{n-1,k}.
$$
But, applying the lemma $0.5$, it is just necessary to prove the
special case $m=1$. Namely, $$a_{n,k}=a_{n-1,k-1}+2\lambda
a_{n-1,k},$$\\
Now it can be easily seen that,
$$a_{n,k}=\sum_{k-1}^{l-1}(2\lambda)^{n-1-l}a_{l,k-1}$$
\end{proof}

\section{Applications}

In this section, we present two applications of vertically-recurrent matrices in
the area of \emph{integral matrices} and  
\emph{electrical engineering}
.
\\
As our first application, we mention an interesting class of integral  matrices which are called \emph{admissible matrices}.
To do so, we consider infinite matrices $A=(a_{n,k})$, indexed by
$\{0,1,2,\ldots\}$, and denote it by $r_m=\{a_{m,0},a_{m,0},\ldots\}$
the $m$th row.
\begin{defn}
$A=(a_{n,k})$ is called admissible [8] if
\begin{enumerate}
	\item 
 $a_{n,k}=0$
for $n<k$, $a_{n,n}=1$ for all $n$ (that is, $A$ is lower
triangular with main diagonal equal to $1$).

	\item 
$r_m.r_n=(a_{m+n,0})$ for all $m,n$, where
$r_m.r_n=\sum_{k}a_{mk}a_{nk}$ is the usual inner product.
	
\end{enumerate}

\end{defn}
Here, we consider some few examples of these matrices and show
that they are indeed vertically-recurrent matrices.
\\
An interesting theorem in [8], states that all admissible matrices
are characterized by sequence $s_0=b_0$; $s_n=b_n-b_{n-1}, n\geq1$
in which $b_{n}=a_{n+1,n}$. Thus, any sequence
$s=\{s_0,s_1,\ldots,s_n,\ldots\}$ will present an admissible
matrix $A=(a_{n,k})$. 

\begin{prop}
	
Let $A=(a_{n,k})$ ba an \emph{admissible} matrix with $a_{n+1,n} = b_{n}$ for all $n$.
Set $s_{0} = b_{0}$, $s_{1} = b_{1} - b_{0}, \ldots, s_{n} = b_{n} - b_{n-1}, \ldots$. Then, we have 
\begin{eqnarray}\label{keyadmis1}
a_{n,k} & = & a_{n-1,k-1} + s_{k}  a_{n-1,k} + a_{n-1,k+1} \hspace{0.5cm} (n\geq 1) \nonumber\\
a_{0,0} & = & 1, \hspace{1cm} a_{0,k} = 1 \hspace{0.5cm} \textit{for} ~~~(k>0).
\end{eqnarray} 	
Conversely, if $a_{n,k}$ is given by 
the recursion (\ref{keyadmis1}), 
then $(a_{n,k})$ is an \emph{admissible} matrix
with $a_{n+1,n} = 
s_{0} + \cdots + s_{n} $.  
\end{prop}

For example the corresponding admissible
matrix for sequence $s=\{1,1,1,\ldots\}$ is
$$
\left[
\begin{array}{ccccc}
1 & 0 & 0 & 0&0\\ 1 &1& 0 & 0&0\\ 2 & 2 &1& 0&0\\4
&5&3&1&0\\9&12&9&4&1
\end{array}
\right],
$$
and for the sequence $s=\{1,2,2,\ldots\}$ is
$$
\left[
\begin{array}{ccccc}
1 & 0 & 0 & 0&0\\ 1 &1& 0 & 0&0\\ 2 & 3 &1& 0&0\\5
&9&5&1&0\\14&28&20&7&1
\end{array}
\right].
$$
Clearly the first matrix is a vertically recurrence matrix with
associated sequence $\mathbf{\Lambda}=\{1,1,2,4,9,\ldots\}$ and the second
one with $\mathbf{\Lambda}=\{1,2,5,14,\ldots\}$ and these are just the
first columns of the above matrices.

\section{Ladder Networks}

The transfer ratio $T_k~ (k=0, 1, 2, \ldots, n)$ of the output to
input signal (voltage or current) along the network (Figure $3$) is
determined by a polynomial in $x$ of the corresponding degree, in
which $x$ determined by the product of impedance of a longitudinal
branch and
admittance of transversal branch [9]. \\
It can be determined from a solution of the following recurrence
equation,
\begin{eqnarray}
& & a_{k+1}-(2+x)a_k+a_{k-1}=0; \nonumber\\
& & a_1=(1+x)a_0,
\end{eqnarray}
where $a_0$ denotes a known signal at the input port of the first
cell and  $a_k$ is the corresponding signal at the k-port of the
network (e.g., $a_k=V_k$ as shown in Figure $3$).
$$
\unitlength 1mm \linethickness{0.4pt}
\begin{picture}(200.33,35.67)
\put(9.00,10.33){\line(1,0){111.00}}
\put(16.67,12.50){\line(0,-1){2.4}}
\put(14.33,12.67){\framebox(5.00,11.67)[cc]{Y}}
\put(34.33,12.20){\line(0,-1){1.9}}
\put(31.67,12.33){\framebox(5.00,11.67)[cc]{Y}}
\put(51.2,12.50){\line(0,-1){2.4}}
\put(48.67,12.67){\framebox(5.00,11.67)[cc]{Y}}
\put(70.2,12.50){\line(0,-1){2.4}}
\put(67.67,12.67){\framebox(5.00,11.67)[cc]{Y}}
\put(83,12.5){\line(0,-1){1.9}}
\put(80.33,12.33){\framebox(5.00,11.67)[cc]{Y}}
\put(97.8,10.33){\line(0,1){1.7}}
\put(95.00,12.33){\framebox(5.00,11.67)[cc]{Y}}
\put(111.00,10.33){\line(0,1){1.9}}
\put(108.67,12.33){\framebox(5.00,11.67)[cc]{Y}}
\put(119.67,10.33){\line(1,0){16.33}}
\put(127.5,10.33){\line(0,1){1.5}}
\put(125.00,12.00){\framebox(5.00,11.67)[cc]{Y}}
\put(16.67,24.5){\line(0,1){3}} \put(9.3,27.33){\line(1,0){10.8}}
\put(20.33,25.00){\framebox(9.33,4.33)[cc]{Z}}
\put(34.67,27){\line(0,-1){2.8}} \put(30,27.00){\line(1,0){8.5}}
\put(38.67,25){\framebox(8.33,4.33)[cc]{Z}} ~
\put(47.00,27.00){\line(1,0){10.2}}
 \put(51,27){\line(0,-1){2.8}}
\put(57.33,25){\framebox(8.33,4.33)[cc]{Z}}
\put(70,27){\line(0,-1){2.6}}
 \put(65.67,27){\line(1,0){20}}
\put(86.00,25.00){\framebox(8.33,4.33)[cc]{Z}}
\put(94.33,27.00){\line(1,0){20.00}}
 \put(82.67,24){\line(0,1){3}}
\put(114.33,25){\framebox(8.33,4.33)[cc]{Z}}
\put(122.67,27){\line(1,0){14.67}}
 \put(8.33,27.3){\circle*{2}}
\put(8.00,10.00){\circle*{2}} \put(136.67,10.00){\circle*{2}}
\put(136.33,27.3){\circle*{2}} \put(97.50,24){\line(0,1){3}}
\put(111.50,24){\line(0,1){3}} \put(127.5,24){\line(0,1){3}}
\put(7.67,12.00){\vector(0,1){12.67}}
\put(39.00,11.00){\vector(0,1){11.33}}
\put(56.00,11.33){\vector(0,1){11.33}}
\put(74.33,11.67){\vector(0,1){11.33}}
\put(102.00,11.67){\vector(0,1){11.33}}
\put(136.67,12.33){\vector(0,1){12.67}}
\put(15.5,28.67){\mbox{$0$}}
 \put(33.67,28.33){\mbox{$1$}}
\put(50.00,28.00){\mbox{$2$}}
 \put(69.33,27.67){\mbox{$3$}}
\put(77.67,27.67){\mbox{$k-1$}}
 \put(97.00,28.00){\mbox{$k$}}
\put(106.33,27.67){\mbox{$n-1$}}
 \put(125.67,27.33){\mbox{$n$}}
\put(132.33,15.67){\mbox{$V_n$}}
 \put(103.33,15.00){\mbox{$V_k$}}
 \put(76.00,14.00){\mbox{$V_3$}}
 \put(58.00,13.67){\mbox{$V_2$}}
 \put(41.00,12.67){\mbox{$V_1$}}
\put(9.00,13.67){\mbox{$V_0$}}
\end{picture}
$$
\begin{center}
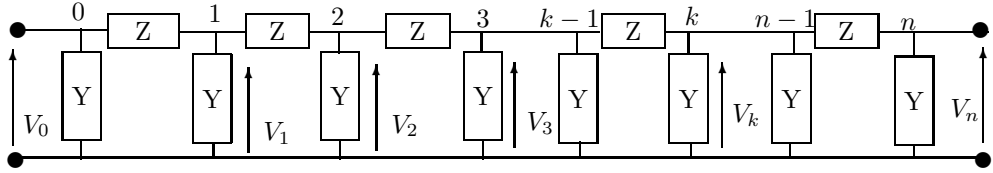
{\bf Figure 3}. Electrical Ladder Network
\end{center}
The ratio $T_k$ follows from the relation
$$
T_k=\frac{a_k}{a_0},~~ k=0, 1, 2, \ldots, n.
$$
It is easy to see that $T_k$ is determined by a polynomial in $x$
of the $k$th degree, so we can write
$$
T_k=\sum_{m=0}^{k}p_{k,m}x^{m},~~ k=0, 1, 2, \ldots, n.
$$
\\
From the direct inspection of the above expression, we have that
\begin{align*}
T_0&=1,\\ T_1&=1+x,\\T_2&=1+3x+x^2,\\T_3&=1+6x+5x^2+x^3,\\
T_4&=1+10x+15x^2+7x^3+x^4,\\T_5&=1+15x+35x^2+28x^3+9x^4+x^5.
\end{align*}
 Now, if we define the matrix $MNT$, \emph{modified numerical
triangle} [9], as follows
$$
(MNT)_{ij}=\left\{
\begin{array}{cc}
p_{i,j} & i\geq j\geq0;\\
0& i<j,
\end{array}
\right.
$$
it is not hard to prove (by mathematical induction) that 
the entries $p_{n,k}$'s have the 
following formula
\begin{equation}
p_{n,k} = {n+k \choose 2k+1}, \hspace{0.5cm} (n \geq k \geq 0).
\end{equation}

Then, we observe that the above array is a vertically-recurrent matrix with associated sequence
$\Lambda=\{\lambda_{n}=n+1\}$. 
Indeed, it is equivalent to prove 
the following \emph{binomial identity}:
\begin{equation}
{n+k \choose 2k+1} = 
\sum_{l=k-1}^{n-1} 
{n-l \choose 1}{l+k-1 \choose 2k-1}
\hspace{0.5cm} (n\geq k \geq 1). 
\end{equation}

For example
$$
\left[
\begin{array}{cccc}
1 & 0 & 0 & 0\\ 1 &1& 0 & 0\\ 1 & 3 &1& 0\\1 &6&5&1
\end{array}
\right]=\left[
\begin{array}{cccc}
1 & 0 & 0 & 0\\ 1 &1& 0 & 0\\ 1 & 2 &1& 0\\1 &3&3&1
\end{array}
\right]\left[
\begin{array}{cccc}
1 & 0 & 0 & 0\\ 0 &1& 0 & 0\\ 0 & 0 &1& 0\\0 &0&2&1
\end{array}
\right]\left[
\begin{array}{cccc}
1 & 0 & 0 & 0\\ 0 &1& 0 & 0\\ 0 & 0 &1& 0\\0 &0&0&1
\end{array}
\right]
$$

It is interesting to note that one can observe that there is also another modified triangle that 
we denote it by 
$MNT_{2}$ which can be defined, as follows

\begin{equation}
(MNT_{2})_{ij}=\left\{
\begin{array}{cc}
{i+2j \choose 3j+1} & i\geq j\geq0;\\
0& i<j,
\end{array}
\right.
\end{equation}

Now, it is easy to see that  $MNT_{2}$ is also a vertically-recurrent 
matrix with associate sequence 
$
\mathbf{\Lambda} = \{ 
\lambda_{n} = {n+2 \choose 2}
\}_{n\geq 0}
$ 
.
Indeed, this claim is equivalent to prove the following combinatorial 
identity:
\begin{equation}
{n+2k \choose 3k+1} = 
\sum_{l=k-1}^{n-1} 
{n-l+1 \choose 2}
{l+2k-2 \choose 3k-2}
\hspace{0.5cm} (n\geq k \geq 1). 
\end{equation}

\section{Open Problems and Conjectures}

Considering the previous discussions, we pose the following open
problems and conjectures. 

\begin{oprob}

Consider the matrix with vertically recurrence relation
$V_n[\Lambda]$ with associated sequence
$\Lambda=\{\lambda_n\}_{n\geq0}$. Find the associated sequence of
the matrix $(V_n[\Lambda])^m$ with respect to the sequence
$\lambda_n$.

\end{oprob}

\begin{oprob}

For any matrix with vertically recurrence sequence $V_n[\Lambda]$
with associated sequence $\Lambda=\{\lambda_n\}_{n\geq0}$
($\lambda_n\in \mathbb{Z};\hspace{.5cm}n=0,1,2,\cdots$), find it's
minimal polynomial in the field of $\mathbb{Z}_{p}$ (see [10]).	
	
\end{oprob}

Let $A=[a_{n,k}]$ be an integral arrays (array with only integer entries) which can be defined recursively, as follows

\begin{eqnarray}
&&\left\{
\begin{array}{cc}
a_{n,k}=\alpha a_{n-1,k-1}+\alpha_{n-1} a_{n-1,k},&~~~(n\geq k\geq1),\\
a_{n,0}=1& n\geq 0,\\
a_{n,k}=0 & k>n,
\end{array}
\right.
\end{eqnarray} 

We also come up with the following conjectures. 

\begin{conj}

The triangular array $A=[a_{n,k}]$ is a vertically-recurrent matrix with associated sequence 
$
\lambda_0 =1 
$
and 
$
\lambda_{i} = \prod_{j=i}^{n} \alpha_{j}
$
.	
	
\end{conj}

\begin{conj}
	
The \emph{Catalan} array $C_{n}$ is a vertically-recurrent 
matrix. 
	 	
\end{conj}

\end{document}